\documentclass[12pt]{amsart}

\usepackage[ansinew]{inputenc}

\usepackage{amscd}
\usepackage{enumitem}
\usepackage{float}

\usepackage{verbatim}

\usepackage{hyperref}
\usepackage{amssymb, amsmath,amsthm}
\usepackage{graphicx}
\usepackage{tikz-cd}
\usepackage{pdfsync}

\newtheorem{fakt}{fakt}[]
\newtheorem{Lemma}[fakt]{Lemma}

\newtheorem{Theorem}[fakt]{Theorem}

\theoremstyle{definition}

\newtheorem{Example}[fakt]{Example}

\newtheorem{Remark}[fakt]{Remark}

\newcommand{\N} {{\mathbb N}} \newcommand{\Z} {{\mathbb Z}}   \newcommand{\C} {{\mathbb C}}

\newcommand{\Spec}{\operatorname{Spec}}
\newcommand{\MaxSpec}{\operatorname{MaxSpec}}

\newcommand{\val}{\operatorname{val}}

\newcommand{\cont}[1]{#1^{\text{cont}}}
\newcommand{\ax}[1]{#1^{\text{ax}}}

\newcommand{\iic}[1]{#1_{>1}}

\newcommand{\ic}[1]{\overline{#1}}
\newcommand{\tight}[1]{#1^{*}}
\newcommand{\Ktight}[2]{{#2}^{{#1}*}}
\newcommand{\sptight}[1]{#1^{*\text{sp}}}

\begin{document}
\title{Tight closure and continuous closure}

\author[Holger Brenner]{Holger Brenner}
\address{Holger Brenner\\
Institut f\"ur Mathematik\\
Universit\"at Osnabr\"uck\\
Albrechtstr. 28a\\
49076 Osnabr\"uck
}
\curraddr{}
\email{hbrenner@uni-osnabrueck.de}
\thanks{}

\author[Jonathan Steinbuch]{Jonathan Steinbuch}
\address{Jonathan Steinbuch\\
Institut f\"ur Mathematik\\
Universit\"at Osnabr\"uck\\
Albrechtstr. 28a\\
49076 Osnabr\"uck
}
\curraddr{}
\email{jonathan.steinbuch@uni-osnabrueck.de}
\thanks{}

\begin{abstract}
We show that for excellent, normal equicharacteristic rings with perfect residue fields the tight closure of an ideal is contained in its axes closure. First we prove this for rings in characteristic $p$.
This is achieved by using the notion of special tight closure established by Huneke and Vraciu.

By reduction to positive characteristic we show that the containment of tight closure in axes closure also holds in characteristic $0$.
From this we deduce that for a normal ring of finite type over $\C$ the tight closure of a primary ideal is inside its continuous closure.
\end{abstract}

\maketitle

\section*{Introduction}

The \emph{tight closure} $\tight{I}$ of an ideal $I$ in a commutative ring $R$ of characteristic $p$ is the ideal consisting of the elements $f\in R$ such that there exists a $c\in R$ not contained in any minimal prime of $R$ (for example a nonzerodivisor) for which we have $cf^q \in I^{[q]}$ for all powers $q=p^e$ of $p$ that are large enough \cite{hochhuntightbrian}.

The \emph{continuous closure} $\cont{I}$ of an ideal $I=(f_1, \ldots , f_n)$ in a finitely generated $\C$-algebra $R$ is the ideal consisting of the elements $f\in R$ such that there are continuous functions $g_1,\ldots, g_n: \Spec{R}(\C)\rightarrow \C$ in the complex topology such that $f=g_1f_1+\ldots+g_nf_n$ holds as an equation of continuous functions \textbf{\cite{Bre06a}}.

In this paper we are concerned with the question whether tight closure is contained in continuous closure.
As continuous closure is a strictly characteristic $0$ concept and tight closure is originally a characteristic $p$ concept, one first has to find the right common ground to compare these notions. 
A first common ground is that both closures are contained in the integral closure of the ideal.
An element $f$ belongs to the integral closure of $I$ if and only if for all ring homomorphisms $\varphi:R \rightarrow V$ to a discrete valuation domain we have $\varphi(x) \in IV$, so for integral closure there is a class of onedimensional test rings.

The \emph{axes closure} $\ax{I}$ of an ideal $I$ in a commutative ring $R$ is the ideal consisting of the elements $f$ such that for any ring homomorphism $\varphi: R \rightarrow S$ to a ring of axes we have $\varphi(f)\in \varphi(I)S$.
A \emph{ring of axes} is a reduced ring whose associated scheme consists of normal curves which intersect transversally in one point, in the sense that the embedding dimension is the number of curves.
One can however take slightly different classes of test rings that yield the same closure as has been studied in detail in \cite[Section 4]{epshoch}.

The axes closure has been introduced in \cite{Bre06a} as an attempt to characterize algebraically the continuous closure though it has failed to exactly do that, since the two closures are different as N. Epstein and M. Hochster have shown \cite[Example 9.2]{epshoch}.
Still they are close: We always have $\cont{I}\subseteq \ax{I}$ and they coincide for primary ideals. An algebraic characterization of continuous closure was given by J. Koll\'{a}r in \cite{kollar}. 

In this paper we prove that the tight closure of an ideal in an excellent, normal ring with perfect residue fields at the maximal ideals is contained in its axes closure. We show this for rings of characteristic $p$ as well as of equal characteristic $0$. The characteristic $0$ case is done by reduction to positive characteristic.

The containment of tight closure inside the continuous closure is strict. A basic result of tight closure theory tells us that tight closure of an ideal in a regular ring is the ideal itself, but in the polynomial ring $\C[X,Y]$ we have $X^2Y^2 \in \cont{ (X^3,Y^3) }$. 

The inclusion result doesn't hold for nonnormal domains. In the ring $R=K[X,Y,Z]/(X^2-YZ^2)$ we have $ \left( \frac{X}{Z} \right)^2 =Y$ and its  normalization is $K[U,Z]$ with $U= \frac{X}{Z} $. Therefore $X \in \tight{(Z)}$ in $R$. The map
\[R \rightarrow K[X,Z]/(X^2-Z^2)=K[X,Z]/(X+Z)(X-Z),\, Y \mapsto 1, X \mapsto X, Z \mapsto Z \] 
shows however that $X$ does not belong to the axes closure of $Z$ and hence for $K=\C$ not to the continuous closure.

\section*{Positive Characteristic}

We prove the main result in positive characteristic using a result of Huneke and Vraciu about the decomposition of tight closure in normal rings. 
Namely, in \cite[Theorem 2.1]{hunekevraciu} they show that for a local, excellent, normal ring of characteristic $p$ the tight closure of any ideal can be written as $$\tight{I}=I+\sptight{I}.$$
Here $\sptight{I}$ denotes the \emph{special tight closure}, which is the ideal of all elements $f$ for which there exists a $q_0>0$ such that $f^{q_0}\in \left(\mathfrak{m}I^{[q_0]}\right)^*$.

We will test axes closure with complete, excellent, local, one-di\-men\-sion\-al, seminormal rings as \cite[Theorem 4.1(3)]{epshoch} says we can do. 
These are very similar to complete axes rings.

In fact, according to \cite[Theorem 3.3]{epshoch}, these rings can be constructed as follows:
Let $\mathfrak{m}$ be the unique maximal ideal of $R$ and $k=R/\mathfrak{m}$ the residue field.
Then $R$ is isomorphic to a subring of the product ring $\prod_{i=1}^{n}V_i$, where $(V_i, \mathfrak{m}_i,L_i)$ are discrete valuation rings whose residue fields are finite extension fields of $k$.

In this product, $R$ is the subring of all elements $(v_1,\ldots, v_n) \in \prod_{i=1}^{n}V_i$ such that the $v_i \mod \mathfrak{m}_i$ are congruent to the same element $\alpha$ in $k$.
The units of $R$ are exactly the elements congruent to a non-zero $\alpha$.

We define functions $\val_i:R\rightarrow \N\cup\{\infty\}, (v_1,\ldots, v_n) \mapsto \val_{V_i}(v_i)$, where $\val_{V_i}$ is the valuation of $V_i$.
Let $\val_i(I):=\min\{\val_i(f) : f\in I\}$.

Over an algebraically closed field of equal characteristic every complete, local, one-dimensional seminormal ring is isomorphic to a complete axes ring and vice versa, a result for which in \cite[Proposition 3.4]{epshoch} Bombieri \cite{bom73} is credited.
For complete axes rings, \cite[Corollary 3.4]{Bre06a} gives a valuative criterion for ideal membership which can be extended with the same methods as follows.

\begin{Lemma}
\label{brecor34}
Let $R$ be a complete, excellent, local, one-dimensional, seminormal ring, $I\subseteq R$ an ideal in $R$ and $f\in R$. If $\val_i(f) > \val_i(I)$ for all $1\leq i\leq n$, then $f\in I$.
\end{Lemma}
\begin{proof}
For every $i\in \{1,\ldots,n\}$ there is, by definition, a $g_i\in I$ with $\val_i(g_i)=\val_i(I)$.
Let $x_i$ be the generator of the maximal ideal $\mathfrak{m}_i$. We can interpret it as an element $(0,\ldots,0,x_i,0,\ldots,0)$ of $R$, where $x_i$ is in the $i$-th position. The product $g_i\cdot x_i$ has order ${\val_i(I)+1}$ in $V_i$. Hence there exists $h_i \in V_i$ such that $f_i =g_ix_ih_i$, where $f_i$ is the $i$th component of $f$. The element $x_ih_i$ has positive order, so its value modulo $\mathfrak{m}_i$ is $0$ and thus there exists the global element
\[ y_i= (0, \ldots, 0, x_ih_i,0 \ldots ,0) \in R \, .\]
So we can write
\[f= (f_1, \ldots, f_n)=  \sum_{i=1}^n g_i y_i \] and hence $f\in I$.
\end{proof}
Note that this is also true if $\val_i(f) = \val_i(I)=\infty$ for some $i$, we thus consider $\infty > \infty$ for the purpose of this Lemma.

\begin{Theorem}\label{pTightInAx}
Let $R$ be an excellent, normal ring of characteristic $p$ such that the residue field at every maximal ideal is perfect.  Let $I$ be an ideal which is primary to a maximal ideal. Then $\tight{I}\subseteq \ax{I}$.
\end{Theorem}
\begin{proof}
Let $I$ be primary to the maximal ideal $\mathfrak m$. \cite[Theorem 4.1(3)]{epshoch} says that we can test the containment in the axes closure on complete, excellent, local, one-dimensional, seminormal rings $S$ with a ring homomorphism $\varphi:R\rightarrow S$. 
Thus we take such a ring and test whether the image of an element $f\in \tight {I}$ is in $IS:=\varphi(I)S$.

Let $\mathfrak{n}$ be the unique maximal ideal of $S$. If $\mathfrak{m} \neq \varphi^{-1}(\mathfrak{n})$ then $\mathfrak{m} \not\subseteq \varphi^{-1}(\mathfrak{n})$ and so $I$ extends to the unit ideal in $S$ and then $f \in IS $. So we can assume that $\mathfrak{m}=\varphi^{-1}(\mathfrak{n})$. 
Then $\varphi$ factors through $R_\mathfrak{m}$. 
We also have $f\in \tight{I}R_\mathfrak{m}\subseteq \tight{(IR_\mathfrak{m})}$ by the persistence of tight closure.  
This means it suffices to show $f\in IS$ in the case that $R$ is local (in effect we exchange $R$ with $R_\mathfrak{m}$) and $R \rightarrow S$ is a local homomorphism. 
This together with the other conditions on $R$ gives us a decomposition $\tight{I}=I+\sptight{I}$ \cite[Theorem 2.1]{hunekevraciu}. 
Thus we write $f=g+h$, where $g\in I$ and $h\in \sptight{I}$. 

By the definition of special tight closure there exists a $q_0>0$ such that $h^{q_0}\in \left(\mathfrak{m}I^{[q_0]}\right)^*$. 

As above we can write $S$ as the subring of a product $\prod_{i=1}^{n}V_i$.
Let $p_i:S\rightarrow V_i$ be the $i$-th projection.
Let us denote $I_i:= I V_i $ and $h_i:=p_i(\varphi(h))$ and $\mathfrak{a}_i = \mathfrak{a} V_i$ with $\mathfrak{a}=\mathfrak{m}S$.
Note that $\val_{V_i}(\mathfrak{a}_i) >0 $ as $\mathfrak{a}_i$ contains no units. 

In a discrete valuation domain we have $J=\ic{J} = J^*$ for all ideals $J$.
Thus,
\[ h^{q_0}\in \left(\mathfrak{m}I^{[q_0]}\right)^* \Rightarrow h_i^{q_0}\in \left(\mathfrak{a}_iI_i^{[q_0]}\right)^* = \mathfrak{a}_iI_i^{[q_0]} \, ,\] by persistence of tight closure \cite[Theorem 1.4.13]{hochhuntight0}, as $S$ and $V_i$ have completely stable weak test elements.
For the valuation $\val:=\val_{V_i}$ on $V_i$ this gives the following inequalities.
\begin{align*}\val(h_i^{q_0}) &\geq \val(\mathfrak{a}_iI_i^{[q_0]})\\
\Rightarrow  q_0\val(h_i) &\geq \val(\mathfrak{a}_i)+q_0\val(I_i)\\
\Rightarrow \val(h_i) &\geq \frac{1}{q_0}\val(\mathfrak{a}_i) + \val(I_i)\\
\Rightarrow \val(h_i) &> \val(I_i).\end{align*}
The last inequality (which might be $\infty > \infty$)  gives \[ \val_i(\varphi(h)) > \val_i(IS) \, . \]
By Lemma \ref{brecor34} the membership $\varphi(h)\in IS$ follows.
Thus $h\in \ax{I}$ and $f=g+h\in \ax{I}$.
\end{proof}

\begin{Theorem}
\label{pTightInAxfinitetype}
Let $I$ be an ideal in a normal domain $R$ of finite type over a perfect field $K$ of characteristic $p$.
Then $  \tight{I}\subseteq \ax{I} $.
\end{Theorem}
\begin{proof}
Let $f \in I^*$. 
We work with ring of axes of finite type over $K$. 
So let $A$ be a seminormal onedimensional ring of finite type over $K$ with only one meeting point corresponding to a maximal ideal $\mathfrak{n} $ of $A$ and let $\varphi:R \rightarrow A$ be a ring homomorphism. 
Then $\mathfrak{m} :=\varphi^{-1}(\mathfrak{n})$ is a maximal ideal of $R$. 
The residue field at $\mathfrak{m}$ is perfect and by the persistence of tight closure we also have $f \in IR_{\mathfrak m}$. 
So we work with the factorization $R_{\mathfrak {m} } \rightarrow A $ and since all conditions for special tight closure hold true in $R_{\mathfrak {m} }$ we can proceed as in the previous proof.
\end{proof}

\begin{Example}
In the proof of the Theorem we use that the special tight closure gives us elements which are in some way deeper inside the tight closure. 
The inner integral closure $\iic{I}$ also measures ``deeper" elements in a similar way and because $\iic{I}\subseteq \ax{I}$ it seems one might think to prove the statement also for inner integral closure instead of axes closure. 
For $\mathfrak{m}$-primary ideals this is possible but in general it doesn't work. 
We have the following simple example which was given to us by Neil Epstein that shows $\sptight{I}\nsubseteq \iic{I}$: 
Let $R=K[[X,Y]], I=(X),\mathfrak{m}=(X,Y)$. 
Then $XY\in \sptight{I}$ (we can put $c=q_o=1$ in the definition of special tight closure), but $XY\notin \iic{I}$.
\end{Example}

\section*{Equal Characteristic 0}

In the following we show a result similar to Theorem \ref{pTightInAxfinitetype} for a ring $R$ of finite type over a field $K$ of characteristic 0.
For this we use the notion of tight closure in characteristic 0 developed in several variants in \cite{hochhuntightbrian} and in more detail in \cite{hochhuntight0}. 

The general idea is that of reduction modulo $p$.
That is, we construct a finitely generated $\Z$-algebra $A$ and a free subalgebra $R_A$ of $R$ over $A$, which tensored with our field $K$ of characteristic 0 becomes the original ring $R$. 
For every prime ideal $\mu$ in $A$ we get a field $\kappa(\mu)=A_\mu/(\mu A_\mu)$.
If we tensor $R_A$ with one of these fields $\kappa:=\kappa(\mu)$ we get $R_\kappa := R_A\otimes_A \kappa$.

We have some freedom in the choice of $A$ which we will use to ensure that $R_A$ itself or the fibers $R_\kappa$ have certain properties (i.e. they hold for $\kappa$ in a Zariski dense open subset of $\MaxSpec A$). 
For deciding whether an element $f \in R$ is in the tight closure of an ideal $I\subseteq R$ we take an algebra $A$ which contains an element $f_A\in R_A$ with $f_A\otimes 1 = f$ in $R_K$ and $I_A\subseteq R_A$, which is meant to say that we have generators for $I$ in $R_A$. 
We say that $u\in \Ktight{K}{I}$ if $u_\kappa \in \tight{I_\kappa}$ in the fibers over a dense open subset of $ \Spec A$. 

We want to check whether an element is in the axes closure and for this we need a suitable version of axes ring.
The finitely generated version for axes rings in \cite[Theorem 4.1(6)]{epshoch} is that of finitely generated étale extensions $S$ of polynomial axes rings $$T=L[X_1,\ldots,X_n]/(X_iX_j,i\neq j)$$ over the algebraic closure $L$ of $K$.

\begin{Theorem}
\label{0TightInAx}
Let $I$ be an ideal in a geometrically normal ring $R$ of finite type over a field $K$ of characteristic $0$. Then $\Ktight{K}{I}\subseteq \ax{I}$.
\end{Theorem}
\begin{proof}
To test whether an element $f\in \Ktight{K}{I}$ is in the axes closure, we take a $K$-algebra homomorphism $\varphi:R\rightarrow T$ to a finitely generated étale extension $T$ of a polynomial axes ring
\[ S=L[X_1,\ldots, X_n]/(X_iX_j,i\neq j) \] over the algebraic closure $L$ of $K$ (these rings characterize axes closure due to \cite[Theorem 4.1(6)]{epshoch}).
We may assume that $T$ is defined and étale over a polynomial ring of axes over a finite extension field $K'$ of $K$.
Moreover, we may assume that $\varphi$ is defined over $K'$.
We replace $K$ by $K'$ and denote it $K$ again.

We perform the reduction described above such that we will get a finitely generated $\Z$-Algebra $A$ and the following diagram of finitely generated $A$-algebras.
\begin{center}
\begin{tikzcd}
R_A \arrow{r}{\varphi_A} & T_A \\
 & S_A \arrow[hookrightarrow]{u}{\psi_A}
\end{tikzcd}
\end{center}
By further shrinking, using the characterization with K\"ahler differentials, we make sure that $\psi_A$ is étale. We also choose $A$ so that $R_\kappa$ will be normal for a dense open subset of fibers $\kappa$ (\cite[Propositon 2.3.17]{hochhuntight0}).

We take a closed fiber $R_\kappa$ of $A\rightarrow R_{A}$ with characteristic $p$ for which $f_\kappa\in \tight{I_\kappa}$.
By Theorem \ref{pTightInAx} we have $f_\kappa\in\ax{I_\kappa}$.

Going from $R_A$ to $R_\kappa$ is a base change, i.e. done by tensoring with $\kappa$. 
Thus the induced morphism $\psi_\kappa$ is étale. 
The inclusion $f_\kappa\in I_\kappa T_{\kappa}$ follows.
%Since the extension to the algebraic closure $\kappa\rightarrow \kappa^{\alg}$ is faithfully flat, we need not worry about the algebraic closedness of $\kappa$ here.

Now assume that $f\notin IT$.
Then $(f,IT)/IT$ is nonzero and free when localized at a single element. 
Thus for a dense open subset of fibers it will still be nonzero. 
But we have shown $f_\kappa\in I_\kappa T_\kappa$ for the fibers over a dense open subset, so we have a contradiction. 
It follows that $f\in IT$.
\end{proof}

\begin{Theorem}
For an ideal $I$ primary to a maximal ideal in a normal, affine $\C$-algebra $R$  we have $\tight{I}\subseteq \cont{I}=\ax{I}$.
\begin{proof}
This follows immediately from \cite[Corollary 7.14]{epshoch} and Theorem \ref{0TightInAx}.
\end{proof}
\end{Theorem}

\begin{Example}
Let $R=K[X,Y,Z]/(X^m+Y^m+Z^m)$, $m \geq 2$, and $I=(X,Y)$. 
Assume that $K$ is algebraically closed and that the characteristic does not divide $m$. Then $Z^2\in (X,Y)^*$ but $Z\notin (X,Y)^*$. 
These are well known results with several proofs. We revover the second part by showing that $Z$ is not in the axes closure of $(X,Y)$. We write 
\[ X^m+Z^m=(X-\xi_1 Z)\cdots (X-\xi_m Z) \]
with some roots of unity $\xi_1,\ldots , \xi_m $.
We have a map 
\[ R\rightarrow S:=K[U,V]/(UV),X\rightarrow \frac{\xi_2U-\xi_1V}{\xi_2-\xi_1},Z\rightarrow \frac{U-V}{\xi_2-\xi_1}, Y\rightarrow 0 \, . \]

This is well defined as we have $U=X-\xi_1 Z$ and $V=X-\xi_2 Z$ and the map sends $X^m+Y^m+Z^m=(X-\xi_1 Z)(X-\xi_2 Z)\cdot Q+Y^m$ to a multiple of $UV$.

Modulo $IS=\left(\frac{\xi_2U-\xi_1V}{\xi_2-\xi_1}\right)$, the ring $S$ becomes $S' \cong K[V]/(V^2)$ but the image of $Z$ in $S'$ is not $ 0$, thus $ZS\notin IS$.

This means, that there is an axes ring for which $Z$ is not in the image of the ideal, thus $Z$ can't be in the axes closure of $I$.
Thus, by Theorem \ref{0TightInAx}, it can also not be in the tight closure, since $R$ is normal.
\end{Example}

\begin{Example}
The containment of tight closure can not be extended to one-dimensional test rings with regular components, we can not drop the condition that the curves meet transversally.
To see this, let  $R=K[X,Y,Z]/(X^3+Y^3-Z^3)$ with a field $K$ of characteristic $\neq 3$ containing a primitive third root of unity $\zeta$ and let again
$I=(X,Y)$. Then $Z^2 \in (X,Y)^*$. However, if we go modulo $Y$ we get
\[A=K[X,Z]/(X^3-Z^3) = K[X,Z](X-Z)(X- \zeta Z)(X- \zeta^2 Z)\]
and its spectrum consists of three lines lying in a plane meeting in one point.
In this ring we have $Z^2 \notin IA=(X)$.
\end{Example}

We conclude with some remarks and questions.

\begin{Remark}
Our main results Theorem \ref{pTightInAxfinitetype} and Theorem \ref{0TightInAx} are also true under the weaker condition that $R$ is a domain of finite type over a field with the property that its seminormalization is already normal.
This rests upon the fact that any ringhomomorphism to an axes ring factors through the seminormalization.
This property means basically that the ring is unibranched in the sense that the completion is a domain.
\end{Remark}

\begin{Remark}
One might ask whether the used result of Huneke and Vraciu is also true under weaker conditions.
Is it true for a complete domain?
Is it true for an excellent local analytically irreducible domain?
Is it true without the assumption that the residue field is perfect?
The example $K[X,Y,Z]/(X^2 - YZ^2)$ localized at $(X,Y-1,Z)$ mentioned in the introduction shows that it can not be true for local domains essentially of finite type without any further assumption.
This follows from the proof of Theorem \ref{pTightInAx}, but can also be seen directly.
We write $W=Y-1$ and work in the ring $S=K[X,W,Z]/(X^2-WZ^2 -Z^2)$ localized at $(X,W,Z)$.
We have $X \in (Z)^*$.
We claim that it is not possible to write $X=ZH + X-ZH$ with $H\in S$ and $X-ZH$ inside the special tight closure of $(Z)$.

To understand the special tight closure of $(Z)$ one has to look at the tight closure $(Z Z^{q_0}, W Z^{q_0}, XZ^{q_0})^*$ for various $q_0=p^{e_0}$. 
The tight closure can be computed in the normalization $K[U,Z]$, the extended ideal is 
\[ (Z^{q_0+1}, (U^2-1) Z^{q_0} , UZ^{q_0+1})=  (Z^{q_0+1}, (U^2-1) Z^{q_0} ) \, . \]
Now assume that $X-ZH \in (Z)^{sp *} $, so
\[  X^q -Z^{q_0} H^{q_0}  \in  (Z^{q_0+1}, (U^2-1) Z^{q_0} )  \,      \]
for some $q_0$.
Then using $X=ZU$ and cancelling with $Z^{q_0} $ gives
\[  U^{q_0} - H^{q_0} \in    (Z, U^2-1) \, . \]
Writing $H=ZA+WB +P(X)$ (with $A,B \in R$ and $P(X)$ a polynomial in $X$) shows that the containment is equivalent to $ U^{q_0} - P_0^{q_0} \in (Z, U^2-1) $, which is for odd characteristic a contradiction.
\end{Remark}

\begin{Remark}
Is tight closure for normal affine $\C$-algebras always inside the continuous closure? 
Is tight closure for normal noetherian rings always inside the axes closure? 
\end{Remark}

\begin{Remark}
As mentioned in the introduction, J. Kollár has given in \cite{kollar} an algebraic characterization of continuous closure, meaning an algebraically defined closure operation for varieties which coincides with the continuous closure if the base field is $\mathbb C$. A natural question is whether tight closure of a normal domain, in particular in positive characteristic, fulfills this algebraic characterization.
\end{Remark}

\begin{Remark}
Is solid closure for normal noetherian rings inside the axes closure? 
Solid closure agrees with tight closure in positive characteristic, but can be strictly larger than tight closure in characteristic zero. 
An example of P. Roberts shows ni \cite{roberts} that in the polynomial ring $K[X,Y,Z]$ over a field $K$ of characteristic zero the inclusion $X^2Y^2Z^2 \in (X^3,Y^3,Z^3)^{\rm sc}$ holds. 
In this example we have indeed the containment in the axes closure and in the continuous closure.
\end{Remark}

\begin{Remark}
The completion of a ring of axes is the completion of a (local version of a) one-dimensional Stanley-Reisner ring. 
Can the inclusion of the tight closure (in the normal case) inside the axes closure be strengthend to an inclusion inside the Stanley-Reisner closure $I^{\rm SR}$ of an ideal?
This closure is defined by taking up to completion the Stanley-Reisner rings as a test category. 
Note that $I \subseteq I^{\rm SR} \subseteq I^{\rm ax} \cap I^{\rm reg}$ and $  I^{*} \subseteq I^{\rm ax} \cap I^{\rm reg}$. 
If we consider Stanley-Reisner rings where the sheets always meet in one point, then the arguments used in this paper go through. 
The main problem is to find a replacement for Lemma \ref{brecor34} in this setting.
\end{Remark}

\begin{Remark}
Is there a reasonable subclass of continuous functions which defines a tight closure type theory for normal $\C$-algebras of finite type? 
In particular, for a smooth variety it should not change the ideals.
\end{Remark}

%\bibliography{Brennersteinbuchtightcontinuous}

\begin{thebibliography}{1}
	
	\bibitem{bom73}
	Enrico Bombieri.
	\newblock Seminormalità e singolarità ordinarie.
	\newblock In {\em Symposia Mathematica}, volume~XI, pages 205--210. Academic
	Press, 1973.
	
	\bibitem{Bre06a}
	Holger Brenner.
	\newblock Continuous solutions to algebraic forcing equations.
	\newblock arXiv:math/0608611v2, 2006.
	
	\bibitem{epshoch}
	Neil Epstein and Melvin Hochster.
	\newblock Continuous closure, axes closure, and natural closure.
	\newblock arXiv:1106.3462, 2011.
	\newblock Accepted to appear in Transactions of the American Mathematical
	Society.
	
	\bibitem{egaIV1}
	Alexander Grothendieck and Jean Dieudonné.
	\newblock {\em Éléments de géométrie algébrique IV. Étude locale des schémas et
		des morphismes de schémas (Première Partie)}, volume~20 of {\em Publications
		mathématiques de l'I.H.É.S}.
	\newblock 1964.
	
	\bibitem{hochhuntightbrian}
	Melvin Hochster and Craig Huneke.
	\newblock Tight closure, integral closure, and the Briançon-Skoda-theorem.
	\newblock {\em Journal of the American Mathematical Society}, 3(1):31--116,
	1990.
	
	\bibitem{hochhuntight0}
	Melvin Hochster and Craig Huneke.
	\newblock Tight closure in equal characteristic 0.
	\newblock In preparation, 1999.
	
	\bibitem{hunekevraciu}
	Craig Huneke and Adela Vraciu.
	\newblock Special tight closure.
	\newblock {\em Nagoya Mathematical Journal}, 170:175--183, 2003.
	
	\bibitem{kollar}
	János Kollár.
	\newblock Continuous closure of sheaves.
	\newblock {\em Michigan Mathematical Journal}, 61:475--491, 2012.
	
	
	
	
	\bibitem{roberts}
	Paul Roberts.
	\newblock  A computation of local cohomology.
	\newblock {\em Contemp. Math.}, 159:351--356, 1994.
	
	

	
	
	\bibitem{stacks-project}
	The {Stacks Project Authors}.
	\newblock \itshape stacks project.
	\newblock \url{http://stacks.math.columbia.edu}, 2017.
	
\end{thebibliography}

\bibliographystyle{plain}

\end{document}